\documentclass{amsart}
\usepackage{amssymb}
\pdfoutput=1
\usepackage{graphicx}
\usepackage{amscd}

\newtheorem{theorem}{Theorem}[section]

\newtheorem{conjecture}[theorem]{Conjecture}

\theoremstyle{definition}
\newtheorem{definition}[theorem]{Definition}

\theoremstyle{remark}
\newtheorem{remark}[theorem]{Remark}

\numberwithin{equation}{section}

\newcommand{\abs}[1]{\lvert#1\rvert}

\begin{document}

\title[An Amazing Prime Heuristic]{An Amazing Prime Heuristic}
\copyrightinfo{2021}{Chris K. Caldwell}
\author{Chris K. Caldwell}
\address{Department of Mathematics and Statistics\\
University of Tennessee at Martin}
\email{caldwell@utm.edu}
%\urladdr{http://www.utm.edu/\symbol{126}caldwell/}
% \thanks{}
\date{March 2021}
% \subjclass{Primary 05C38, 15A15; Secondary 05A15, 15A18}
% \keywords{elementary prime number theory, heuristics, twin primes, Sophie Germain
% primes}
% \dedicatory{Dedicated to my wife and five children}

%\begin{abstract}
%Dickson conjectured that a set of polynomials will take on
%infinitely many simultaneous prime values.  Later others, such as
%Hardy and Littlewood, gave estimates for the number of these
%primes.  In this article we look at this conjecture, develop a
%simple heuristic and rederive these classic estimates.  We then
%apply them to several special forms of primes and compare the
%estimates with the actual numbers.
%\end{abstract}

\maketitle

\section{Introduction}
\addtocontents{toc}{\vspace{3pt}}

The record for the largest known twin prime is constantly
changing.  For example, in October of 2000, David Underbakke
found the record primes:
$$83475759\cdot 2^{64955}\pm 1.$$
The very next day Giovanni La Barbera found the new record primes:
$$1693965\cdot 2^{66443}\pm 1.$$
The fact that the size of these records are close is no
coincidence!  Before we seek a record like this, we usually try to
estimate how long the search might take, and use this information
to determine our search parameters.  To do this we need to know
how common twin primes are.

It has been conjectured that the number of twin primes less than
or equal to $N$ is asymptotic to
$$2C_{2}\int_{2}^{N}\frac{dx}{(\log x)^{2}}\sim
\frac{2C_{2}N}{(\log N)^{2}}$$ where $C_2$, called the twin prime
constant, is approximately $0.6601618$.  Using this we can
estimate how many numbers we will need to try before we find a
prime.  In the case of Underbakke and La Barbera, they were both
using the same sieving software (NewPGen\footnote{Available from
http://www.utm.edu/research/primes/programs/NewPGen/} by Paul
Jobling) and the
same primality proving software (Proth.exe\footnote{Available from
http://www.utm.edu/research/primes/programs/gallot/} by Yves
Gallot) on similar hardware--so of course they choose
similar ranges to search. But where does this conjecture
come from?

In this chapter we will discuss a general method to form
conjectures similar to the twin prime conjecture above.  We will
then apply it to a number of different forms of primes such as
Sophie Germain primes, primes in arithmetic progressions,
primorial primes and even the Goldbach conjecture.  In each case
we will compute the relevant constants (e.g., the twin prime
constant), then compare the conjectures to the results of computer
searches. A few of these results are new--but our main goal is to
illustrate an important technique in heuristic prime number theory and
apply it in a consistent way to a wide variety of problems.

\subsection{The key heuristic}

A heuristic is an educated guess.  We often use them to give a
rough idea of how long a program might run--to estimate how long
we might need to wait in the brush before a large prime comes
wandering by.  The key to all the results in this chapter is the
following:

\begin{quote}
The prime number theorem states that the number of primes less
than $n$ is asymptotic to $1/\log n$. So if we choose a random
integer $m$ from the interval $[1,n]$, then the probability that
$m$ is prime is asymptotic to $1/\log n$.
\end{quote}

\noindent Let us begin by applting this to a few simple
examples.

First, as $n$ increases, $1/\log n$ decreases, so it seemed reasonable to
Hardy and Littlewood to conjecture that there are more primes in
the set $\{1, 2, 3, \ldots, k\}$ than in $\{n+1, n+2, n+3,
\ldots, n+k\}$.  In other words, Hardy and Littlewood \cite{HL23}
conjectured.

\begin{conjecture}
For sufficiently large integers $n$, $\pi(k) \ge \pi(n+k) -
\pi(n)$.
\label{BadConjecture}
\end{conjecture}

They made this conjecture on the basis of very little numerical
evidence saying ``An examination of the primes less than 200
suggests forcibly that $\rho(x) \le \pi(x) (x \ge 2)$'' (page
54). (Here $\rho(x) = \limsup_{n \rightarrow \infty} \pi(n+x) -
\pi(x)$.)   By 1961 Schinzel \cite{Schinzel61} had verified this to
$k=146$ and by 1974 Selfridge {\it et. al.} \cite{HR74} had
verified it to 500. As reasonable sounding as this conjecture is,
we will give a strong argument against it in just a moment.

Second, suppose the Fermat numbers $F_n = 2^{2^n}+1$ behaved as
random numbers.\footnote{There are reasons not to assume this such as
the Fermat numbers are pairwise relatively prime.}
Then the probability that $F_n$
is prime should be about $1/\log(F_n) \sim 1/(2^n \log 2)$.  So
the expected number of such primes would be on the order of
$\sum_{n=0}^{\infty} 1/(2^n \log 2) = 2/\log 2$.  This is the
argument Hardy and Wright used when they presented the following
conjecture \cite[pp. 15, 19]{HW79}:
\begin{conjecture} \label{FermtPrimeConjecture}
There are finitely many Fermat primes.
\end{conjecture}

The same argument, when applied to the Mersenne numbers, Woodall
numbers, or Cullen numbers suggest that there are infinitely many
primes of each of these forms.  But it would also imply there are
infinitely many primes of the form $3^n-1$, even though all but
one of these are composite.  So we must be a more careful than
just adding up the terms $1/\log n$.  We will illustrate how this
might be done in the case of polynomials in the next section.

As a final example we point out that in 1904, Dickson conjectured
the following:

\begin{conjecture} \label{DicksonsConjecture}
Suppose $a_i$ and $b_i$ are integers with $a_i>1$.
If there is no prime which
divides each of
$$\{b_1 x + a_1, b_2 x + a_2, \ldots, b_n x + a_n \}$$
for every $x$, then there are infinitely many integers values of
$x$ for which these terms are simultaneously prime.
\end{conjecture}

How do we arrive at this conclusion?  By our heuristic, for
each $x$ the number $b_i x + a_i$ should be prime with a
probability $1/\log N$.  If the probabilities that each term is
prime are independent, then the whole set should be prime with
probability $1/(\log N)^n$.  They are not likely to be
independent, so we expect something on the order of $C/(\log N)^n$
for some constant $C$ (a function of the coefficients $a_i$ and
$b_i$).

In the following section we will sharpen Dickson's conjecture
in to a precise form like that of the twin prime conjecture above.

\subsection{A warning about heuristics}

What (if any) value do such estimates have?

They may have a great deal of value for those searching for
records and predicting computer run times, but mathematically
they have very little value.  They are just (educated) guesses,
not proven statements, so not ``real mathematics.'' Hardy and
Littlewood wrote: ``{\it Probability} is not a notion of pure
mathematics, but of philosophy or physics'' \cite[pg 37 footnote
4]{HL23}. They even felt it necessary to apologize for their
heuristic work:

\begin{quotation}
Here we are unable (with or without Hypothesis $R$) to offer
anything approaching a rigorous proof.  What our method yields is a
{\it formula}, and one which seems to stand the test of
comparison with the facts.  In this concluding section we propose
to state a number of further formulae of the same kind.
Our apology for doing so must be (1) that no similar formulae have
been suggested before, and that the process by which they are deduced
has at least a certain algebraic interest, and (2) that it seems
to us very desirable that (in default of proof)
the formula should be checked, and that we hope that some of the
many mathematicians interested in the computative
side of the theory of numbers may find them worthy
of their attention.  (\cite[pg 40]{HL23})
\end{quotation}

\noindent When commenting on this Bach and Shallit wrote:

\begin{quotation}
Clearly, no one can mistake these probabilistic arguments for
rigorous mathematics and remain in a state of
grace.\footnote{Compare this quote to John von Neumann's remark in
1951 ``Anyone who considers arithmetical methods of producing
random digits is, of course, in a state of sin.'' \cite[p.
1]{Knuth75}} Nevertheless, they are useful in making educated
guesses as to how number-theoretic functions should ``behave.''
(\cite[p. 248]{BS96})
\end{quotation}

Why this negative attitude?  Mathematics seeks proof.  It requires
results without doubt or dependence on unnecessary assumptions.
And to be blunt, sometimes heuristics fail!  Not only that, they
sometimes fail for even the most cautious of users. In fact we
have already given an example (perhaps you noticed?)

Hardy and Littlewood, like Dickson, conjectured that if there is
no prime which divides each of the following terms for every $x$,
then they are simultaneously primes infinitely often:
$$\{x+a_1, x+a_2, x+a_3, x+a_4, x+a_5, \ldots,  x+a_k\}$$
\cite[Conjecture X]{HL23}.  This is a special case of Dickson's
Conjecture is sometimes called {\bf the prime $k$-tuple conjecture}.
We have also seen that they conjectured $\pi(k) \ge \pi(n+k) -
\pi(n)$ (conjecture \ref{BadConjecture}). But in 1972, Douglas
Hensley and Ian Richards proved that one of these two conjectures
is false \cite{HR72, HR73, Richards74}!

Perhaps the easiest way to see the conflict between these
conjectures is to consider the following set of polynomials
found by Tony Forbes \cite{Forbes95}:

\begin{eqnarray*} \{n-p_{24049}, n-p_{24043}, \ldots, n-p_{1223}, n-p_{1217},\\
n+p_{1217}, n+p_{1223}, \ldots, n+p_{24043}, n+p_{24049}\}
\end{eqnarray*}

\noindent where $p_n$ is the $n$-th prime.  By Hardy and
Littlewood's first conjecture there are infinitely many integers $n$ so
that each of these $4954$ terms are prime.  But the width of this
interval is just $48098$ and $\pi(48098)<4954$.  So this
contradicts the second conjecture.

If one of these conjectures is wrong, which is it? Most
mathematicians feel it is the second conjecture that $\pi(k) \ge
\pi(n+k) - \pi(n)$ which is wrong. The prime $k$-tuple conjecture
receives wide support (and use!)  Hensley and Richards predicted
however
\begin{quotation}
Now we come to the second problem mentioned at the beginning of
this section: namely the smallest number $x_1+y_1$ for
which $\pi(x_1+y_1) > \pi(x_1)+\pi(y_1)$. We suspect, even
assuming the $k$-tuples hypothesis (B) is eventually proved
constructively, that the value of $x_1+y_1$ will never be found;
and moreover that no pair $x,y$ satisfying $\pi(x+y) >
\pi(x)+\pi(y)$ will ever be computed. (\cite[p. 385]{HR74})
\end{quotation}

What can we conclude from this example of clashing conjectures?
First that heuristics should be applied only with care.  Next they
should then be carefully tested. Even after great care and testing
you should not stake to much on their predictions, so
read this chapter with the usual bargain hunter's mottoes in
mind: ``buyer beware'' and ``your mileage may vary.''

\subsection{Read the masters}
%[[ Chris, write something here.  Here is a start. ]]

The great mathematician Abel once wrote "It appears to me that if
one wants to make progress in mathematics, one should study the
masters and not the pupils."  Good advice,  but this is an area
short of masters.

Hardy and Littlewood's third paper on their circle method
\cite{HL23} is one of the first key papers in this area.  In this paper
they made the first real step toward the proving the Goldbach
conjecture, then gave more than a dozen conjectures on the
distribution of primes.  Their method is far more complicated
than what we present here--but it laid the basis for actual
proofs of some related results.

The approach we take here may have first been laid out by
Cherwell and Wright \cite[section 3]{CW60} (building on earlier
work by Cherwell \cite{Cherwell46}, St{\"a}ckel
\cite{Stackel1916}, and of course Hardy and Littlewood).  The
same approach was taken by Bateman and Horn \cite{BH62} (see also
\cite{BH65}).

Many authors give similar arguments including Brent, Shanks
\cite{Shanks62,SK67,Shanks78} and P{'o}lya \cite{Polya59}.

There are also a couple excellent ``students'' we should mention.
Ribenboim included an entire chapter on heuristics in his text
``the new book of prime number records'' \cite{Ribenboim95}.
Riesel also develops much of this material in his fine book
\cite{Riesel94}.  See also Schroeder
\cite{Schroeder83,Schroeder97}.

And as enthusiastic students we also add our little mark. Again,
most of what we present here was first done by others. Our only
claim to fame is a persistent unrelenting application of one
simple idea to a wide variety of problems. Enough talk, let's get
started!

%%%%%%%%%%%%%%%%%%%%%%%%%%%%%%%%%%%%%%%%%%%%%%%%
%%
%%  The prototypical example
%%
%%%%%%%%%%%%%%%%%%%%%%%%%%%%%%%%%%%%%%%%%%%%%%%%

\addtocontents{toc}{\vspace{5pt}}
\section{The prototypical example: sets of polynomials}
\addtocontents{toc}{\vspace{3pt}}

\subsection{Sets of polynomials}

We regularly look for integers that make a set of (one or more)
polynomials simultaneously prime. For example, simultaneous prime
values of $\{n,n+2\}$ are twin primes, of $\{n,2n+1\}$ are Sophie
Germain primes, and of $\{n,2n+1,4n+3\}$ are Cunningham chains of
length three.  So this is an interesting test case for our
heuristic.

What might stop a set of integer valued polynomials from being
simultaneously prime?  The same things that keep a single
polynomial from being prime: It might factor like
$9x^2-1$, or there might be a prime which divides every
value of the polynomial such as $3$ and $x^3-x+9$.  So
before we go much further we need a few restrictions on our
polynomials $f_{1}(x)$, $f_{2}(x)$, \ldots, $f_{k}(x)$. We
require that

\begin{itemize}
\item the polynomials $f_{i}(x)$ are irreducible, integer valued,
and have positive leading terms, and
\item  the degree $d_{i}$ of $f_{i}(x)$ is greater than zero $(i=1,2,...,k)$.
\end{itemize}

If we could treat the values of these polynomials at $n$ as
independent random variables, then by our key heuristic, the
probability that they would be simultaneously prime at $n$ would
be

\begin{equation}
\prod\limits_{i=1}^{k}\frac{1}{\log \text{ }f_{i}(n)}\sim
\frac{1}{ d_{1}d_{2}...d_{k}(\log n)^{k}}. \label{p(independent)}
\end{equation}

\noindent So the number of values of $n$ with $0<n\leq N$ which
yield primes would be primes approximately \

\begin{equation*}
\frac{1}{d_{1}d_{2}...d_{k}}\int_{2}^{N}\frac{dx}{(\log
x)^{k}}\sim \frac{N}{d_{1}d_{2}...d_{k}(\log N)^{k}}.
\end{equation*}

However, the polynomials are unlikely to behave both randomly and
independently. For example, $\{n,n+2\}$ are either both odd or
both even; and the second of $\{n,2n+1\}$ is never divisible by
two. To attempt to adjust for this, for each prime $p$, we will
multiply by the ratio of the probability that $p$ does not divide
the product of the polynomials at $n$, to the probability that
$p$ does not divides one of $k$ random integers. In other words,
we will adjust by multiplying by a measure of how far from
independently random the values are.

To find this {\em adjustment factor}, we start with the following
definition:

\begin{definition}
$w(p)$ is the number of solutions to $f_{1}(x)f_{2}(x)\cdot
...\cdot f_{k}(x)\equiv 0 \pmod{p}$ with $x$ in
$\{0,1,2,...,p-1\}$.
\end{definition}
For each prime then, we need to multiply by
\begin{equation}
\frac{\frac{p-w(p)}{p}}{(\frac{p-1}{p})^{k}}=\frac{1-w(p)/p}{(1-1/p)^{k}},
\label{factor for prime}
\end{equation}
and our complete adjustment factor is found by taking the product
over all primes $p$:
\begin{equation}
\prod\limits_{p}\frac{1-w(p)/p}{(1-1/p)^{k}}.  \label{adjustment
factor}
\end{equation}
This gives us the following conjecture (see
\cite{BH62,Dickson04}).

\begin{conjecture}[Dickson's Conjecture]
Let the irreducible polynomials $f_1 (x),$ $f_2 (x),$ \ldots, $f_k
(x)$ be integer valued, have a positive leading term, and suppose
$f_i(x)$ has degree $d_i >0$ $(i=1,2,\ldots,k)$. The number of
values of $n$ with $0 < n \leq N$ which yield simultaneous primes
is approximately

\begin{equation}
\frac{1}{d_{1}d_{2}...d_{k}}\prod\limits_{p}\frac{1-w(p)/p}{(1-1/p)^{k}}%
\int_{2}^{N}\frac{dx}{(\log x)^{k}}\sim
\frac{N}{d_{1}d_{2}...d_{k}(\log
N)^{k}}\prod\limits_{p}\frac{1-w(p)/p}{(1-1/p)^{k}}.
\label{pi(general)}
\end{equation}
\label{KeyConjecture}
\end{conjecture}

The ratio on the right is sufficient if $N$ is very large or we just
need a rough estimate, but the integral usually gives a better
estimate for small $N$.  Sometimes we wish an even tighter
estimate for relatively small $N$. Then we use the right side of
equation \ref{p(independent)} and write the integral in the
conjecture above as

\begin{equation}\label{pi(better)}
\prod\limits_{p}\frac{1-w(p)/p}{(1-1/p)^{k}}
\int_{2}^{N}\frac{dx}{\log f_1(x) \log f_2(x)\cdot\ldots\cdot\log f_k(x)
 }
\end{equation}

\addtocontents{toc}{\vspace{5pt}}
\section{Sequences of linear polynomials}
\addtocontents{toc}{\vspace{3pt}}

Conjecture \ref{KeyConjecture} gives us an approach to estimating
the number of primes of several forms. In this section we will
apply conjecture it to twin primes, Sophie
Germain primes, primes of the form $n^{2}+1$, and several other
forms of primes. In each case, we will compare the estimates in
the conjecture to the actual numbers of such primes.

\subsection{Twin primes}

To find twin primes we can use the polynomials $n$ and $n+2$.
Note that $w(2)=1$, and $w(p)=2$ for all odd primes $p$. With a
little algebra, we see our adjustment factor \ref{adjustment
factor} is

\begin{equation}\label{eq:C_2}
2\prod\limits_{p>2}1-\frac{1}{(p-1)^{2}} =
2\prod\limits_{p>2}\frac{p(p-2)}{(p-1)^{2}}.
\end{equation}

\noindent This product over odd primes is called the twin primes constant:
$$C_{2}=0.66016\ 18158\ 46869\ 57392\ 78121\ 10014\ 55577\ 84326\ ...$$

\noindent Gerhard Niklasch has computed $C_2$ to over 1000 decimal
places using the methods of Moree \cite{Moree2000}.

In this case, conjecture \ref{KeyConjecture} becomes:

\begin{conjecture}[Twin prime conjecture] \label{TwinPrimeConjecture}
The expected number of twin primes $\{p,p+2\}$ with $p\leq N$ is
\begin{equation}\label{pi(twin)}
2C_{2}\int_{2}^{N}\frac{dx}{(\log x)^{2}}\sim
\frac{2C_{2}N}{(\log N)^{2}}.
\end{equation}
\end{conjecture}

\noindent (This is \cite[Conjecture ??]{HL14}.)  For a different
heuristic argument for the same result see \cite[section 22.20]{HW79}.

In practice this seems to be a exceptionally good estimate (even
for small $N$)--see Table \ref{table:Twin}.  (The last few values
in Table \ref{table:Twin} were calculated by T. Nicely
\cite{Nicely95}.\footnote{See also
http://www.trnicely.net/counts.html}.)

\begin{table}
\caption{Twin primes less than $N$}
\begin{tabular}{crrr}
\hline & \textbf{actual} &
\multicolumn{2}{c}{\textbf{predicted}} \\
$N$ & \textbf{number} & \textbf{integral} & \textbf{ratio}
\\ \hline
$10^{3}$ & \textbf{35} & \textbf{46} & 28 \\
$10^{4}$ & \textbf{205} & \textbf{214} & 155 \\
$10^{5}$ & \textbf{1224} & \textbf{1249} & 996 \\
$10^{6}$ & \textbf{8169} & \textbf{8248} & 6917 \\
$10^{7}$ & \textbf{58980} & \textbf{58754} & 50822 \\
$10^{8}$ & \textbf{440312} & \textbf{440368} & 389107 \\
$10^{9}$ & \textbf{3424506} & \textbf{3425308} & 3074426 \\
$10^{10}$ & \textbf{27412679} & \textbf{27411417} & 24902848 \\
$10^{11}$ & \textbf{224376048} & \textbf{224368865} & 205808661 \\
$10^{12}$ & \textbf{1870585220} & \textbf{1870559867} & 1729364449 \\
$10^{13}$ & \textbf{15834664872} & \textbf{15834598305} & 14735413063 \\
$10^{14}$ & \textbf{135780321665} & \textbf{135780264894} & 127055347335 \\
$10^{15}$ & \textbf{1177209242304} & \textbf{1177208491861} & 1106793247903 \\
%$10^{16}$ & \textbf{??} & \textbf{10304192554496} & 9727675030402 \\
%$10^{17}$ & \textbf{??} & \textbf{90948833260990} & 86169024490758 \\
\hline
\end{tabular}
\label{table:Twin}
\end{table}

It has been proven by sieve methods, that if you replace the $2$
in our estimate (\ref{pi(twin)}) for the number of twin primes
with a $7$, then you have a provable upper bound for $N$
sufficiently large. Brun first took this approach in 1919 when he
showed we could replace the $2$ with a $100$ and get an upper
bound from some point $N_{0}$ onward \cite{Brun19}. There has
been steady progress reducing the constant since Brun's article
(and $7$ is not the current best possible value).  Unfortunately
there is no known way of changing this into a lower bound--as it
has not yet been proven there are infinitely many twin primes.

\subsection{Prime pairs $\{n,n+2k\}$ and the Goldbach conjecture}

What if we replace the polynomials $\{n,n+2\}$ with $\{n,n+2k\}$?
In this case $w(p)=1$ if $p|2k$ and $w(p)=2$ otherwise, so the
adjustment factor \ref{eq:C_2} becomes
\begin{equation}
C_{2,k} = C_{2}\prod\limits_{p\mid k, p>2}\frac{p-1}{p-2}.
\end{equation}
With this slight change, conjecture \ref{KeyConjecture} now is

\begin{conjecture}[Prime pairs conjecture]
The expected number of prime pairs $\{p,p+2k\}$ with $p\leq N$ is
\begin{equation}
2C_{2,k}\int_{2}^{N}\frac{dx}{(\log x)^{2}}\sim
\frac{2C_{2,k}N}{(\log N)^{2}}. \label{eq:PrimePairs}
\end{equation}
\label{conj:PrimePairs}
\end{conjecture}
\noindent (This is \cite[Conjecture B]{HL23}.)

For example, when searching for primes $\{n,n+210\}$ we expect to
find (asymptotically) $\frac{2}{1}\frac{4}{3}\frac{6}{5}=3.2$
times as many primes as we find twins.  Table \ref{table:pairs}
shows that this is indeed the case.

\begin{table}[htbp]
\caption{Prime pairs $\{n,n+2k\}$ with $n\le N$}
\begin{tabular}{rrr|rr|rr}
\hline
 & \multicolumn{2}{c}{$k=6$} & \multicolumn{2}{c}{$k=30$} & \multicolumn{2}{c}{$k=210$} \\
$N$ & \textbf{actual} & \textbf{predicted} & \textbf{actual} & \textbf{predicted}
 & \textbf{actual} & \textbf{predicted} \\
\hline $10^3$ & \textbf{74} & \textbf{86} & \textbf{99} &
\textbf{109}
 & \textbf{107} & \textbf{118} \\
$10^4$ & \textbf{411} & \textbf{423} & \textbf{536} & \textbf{558}
 & \textbf{641} & \textbf{653} \\
$10^5$ & \textbf{2447} & \textbf{2493} & \textbf{3329} &
\textbf{3316}
 & \textbf{3928} & \textbf{3962} \\
$10^6$ & \textbf{16386} & \textbf{16491} & \textbf{21990} &
\textbf{21981}
 & \textbf{26178} & \textbf{26358} \\
$10^7$ & \textbf{117207} & \textbf{117502} & \textbf{156517}
 & \textbf{156663} & \textbf{187731} & \textbf{187976} \\
$10^8$ & \textbf{879980} & \textbf{880730} & \textbf{1173934}
 & \textbf{1174300} & \textbf{1409150} & \textbf{1409141} \\
$10^9$ & \textbf{6849047} & \textbf{6850611} & \textbf{9136632}
 & \textbf{9134141} & \textbf{10958370} & \textbf{10960950} \\
\hline
\end{tabular}
\label{table:pairs}
\end{table}

Note that asymptotically equation \ref{eq:PrimePairs} must also
give the expected number of consecutive primes whose difference
is $k$.  This can be shown (and the values estimated more
accurately for small $N$) using the inclusion-exclusion principle
\cite{Brent74,Mayoh68}.  From this it is conjectured that the
most common gaps between primes $\le N$ is always either 4 or a
primorial number ($2, 6, 30, 210, 2310, \ldots$) \cite{Harley94}.

``But wait--there is more'' the old infomercial exclaimed ``it
dices, it slices...''  Look at the prime pairs set this way: $\{
n, 2k-n\}$.  Now when both terms are prime we have found two
primes which add to $2k$.  Our adjustment factor is unchanged, so
the number of ways of writing $2k$ as a sum of two primes, often
denoted $G(2k)$, is approximately:
\begin{equation}
G(2k) \sim 2C_{2,k}N \int_{2}^{N}\frac{dx}{\log x \log(2k-x)}.
\end{equation}
This is equivalent to the conjecture as given by Hardy and
Littlewood \cite[Conjecture A]{HL23}:
\begin{conjecture}[Extended Goldbach conjecture]
The number of ways of writing $2k$ as a sum of two primes is
asymptotic to
\begin{equation} 2C_{2,k}\int_{2}^{N}\frac{dx}{(\log x)^{2}}\sim
\frac{2C_{2,k}N}{(\log N)^{2}}.
\end{equation}
\end{conjecture}

Hardy and Littlewood suggest that for testing this against the
actual results for small numbers, we follow Shah and Wilson and
use $1/((\log N)^2 - \log N)$ instead of $1/(\log N)^2$.

\subsection{Primes in Arithmetic Progression}

The same reasoning could be applied to estimate the number of
arithmetic progressions of primes with length $k$ by seeking
integers $n$ and $k$ such that each term of
$$\{n, n+d, n+2d, \ldots, n+(k-1)d\}$$ is prime.  In this case $w(p)=1$ if
$p$ divides $d$, and $w(p) = \min(p,k)$ otherwise. In particular,
if we wish all of the terms to be primes we must have $p|d$ for
all primes $p \le k$.  When this is the case, for a fixed $d$ we
have

\begin{equation}\label{eq:Akd}
A_{k,d} = \prod\limits_{p | d}\frac{1}{(1-1/p)^{k-1}}
\prod\limits_{p \nmid d}\frac{1-k/p}{(1-1/p)^{k}}.
\end{equation}

\noindent We can rewrite these in terms of the
Hardy-Littlewood constants

\begin{equation}\label{eq:c_k}
c_k = \prod\limits_{p > k}\frac{1-k/p}{(1-1/p)^{k-1}}
\end{equation}

\noindent as follows

$$A_{k,d} = c_k \prod\limits_{p \leq k}\frac{1}{(1-1/p)^{k-1}}
\prod\limits_{p > k, p \mid d}\frac{p-1}{p-k}.$$

\noindent Of course $A_{k,d}=0$ if $k\#$ does not divide $d$.

It is possible to estimate $c_k$ and $A_{k,k\#}$ to a half dozen
significant digits using product above over the first several
hundred million primes--but at the end of this section we will
show a much better method.  Table \ref{table:A} contains
approximations of the first of these constants.

\begin{table}[htbp]
\caption{Adjustment factors $A_{k,k\#}$ for arithmetic sequences}
\begin{tabular}{rrl|rrl} \hline
$k$ & $k\#$ & $A_{k,k\#}$ & $k$ & $k\#$ & $A_{k,k\#}$ \\
\hline
2 & 2 &   1.32032363169374 & 11 & 2310   & 629913.461423349 \\
3 & 6 &   5.71649719143844 & 12 & 2310   & 1135007.50238685 \\
4 & 6 &   8.30236172647483 & 13 & 30030  & 45046656.1742087 \\
5 & 30 &  81.0543595999686 & 13 & 30030  & 132128113.722194 \\
6 & 30 &  138.388898492679 & 15 & 30030  & 320552424.308155 \\
7 & 210 & 2590.65351840622 & 16 & 30030  & 527357440.662591 \\
8 & 210 & 7130.47817586170 & 17 & 510510 & 23636723084.1607 \\
9 & 210 & 16129.6476839631 & 18 & 510510 & 47093023670.0967 \\
10 & 210& 24548.2695388318 & 19 & 9699690& 3153485401596.08 \\
\hline
\end{tabular}
\label{table:A}
\end{table}

Once again we reformulate conjecture \ref{KeyConjecture} for our
specific case and this time find the following.

\begin{conjecture}
The number of arithmetic progressions of primes with length $k$,
common difference $d$, and beginning with a prime $p\leq N$ is
\begin{equation}
A_{k,d}\int_{2}^{N}\frac{dx}{(\log x)^{k}} \sim
\frac{A_{k,d}N}{(\log N)^{k}}.
\end{equation}
\end{conjecture}
To check this conjecture we counted the number of such arithmetic
progressions with common differences $6$, $30$, $210$ and $2310$.
The results (table \ref{table:AP345}) seem to support this
estimate well.

\begin{table}[htbp]
\caption{Primes in arithmetic progression, starting before $10^9$}
\begin{tabular}{rrr|rr|rr}
\hline common & \multicolumn{2}{c|}{length $k=3$} &
\multicolumn{2}{c|}{length $k=4$} & \multicolumn{2}{c}{length $k=5$} \\
difference & actual & predicted &
actual & predicted & actual & predicted \\
\hline
6 & \textbf{758163} & \textbf{759591} & \textbf{56643} & \textbf{56772} & \textbf{0} & \textbf{0} \\
30 & \textbf{1519360} & \textbf{1519170} & \textbf{227620} & \textbf{227074} & \textbf{28917} & \textbf{28687}  \\
210 & \textbf{2276278} & \textbf{2278725} & \textbf{452784} & \textbf{454118} & \textbf{85425} & \textbf{86037} \\
2310 & \textbf{2847408} & \textbf{2848284} & \textbf{648337} & \textbf{648640} & \textbf{142698} & \textbf{143326} \\
\hline & \multicolumn{2}{c|}{length $k=6$} &
\multicolumn{2}{c|}{length $k=7$} & \multicolumn{2}{c}{length
$k=8$} \\
\hline
30 & \textbf{2519} & \textbf{2555} & \textbf{0} & \textbf{0} & \textbf{0} & \textbf{0} \\
210 & \textbf{15146} & \textbf{15315} & \textbf{2482} & \textbf{2515} & \textbf{353} & \textbf{370}\\
2310 & \textbf{30339} & \textbf{30588} & \textbf{6154} & \textbf{6266} & \textbf{1149} & \textbf{1221} \\
\hline
\end{tabular}
\label{table:AP345}
\end{table}

This conjecture also includes some of the previous results as
special cases.  For example, when $k$ is one, we are just counting
primes, and as expected, $A_{1,d}=1$. It is also easy to show that
$A_{2,d}=2C_{2,d}$ and $A_{2,2}=2C_2$, so $A_{2,d}$ matches the
values from the Prime Pairs Conjecture \ref{conj:PrimePairs}.

What if we do not fix the common difference? Instead we might ask
how many arithmetic progressions of primes(with any common
difference) there are all of whose terms are less than $x$. Call
this number $N_k(x)$.  Grosswald \cite{Grosswald82} modified Hardy
\& Littlewoods' Conjecture X to conjecture:

\begin{conjecture}
The number of arithmetic progressions of primes $N_k(N)$ with
length $k$ all of whose terms are less than $N$ is
\begin{equation}
N_k(x) \sim \frac{D_k N^2}{2(k-1)(\log N)^{k}}
\end{equation}
where
\begin{equation}
D_k = \prod_{p \leq k} \frac{1}{p} \left(\frac{p}{p-1}\right)
^{k-1} \prod_{p
> k} \left( \frac{p}{p-1} \right) ^{k-1} \frac{p-k+1}{p}.
\end{equation}
\end{conjecture}

\noindent Grosswald was able to prove this result in the case
$k=3$ \cite{GH79}.  His article also included approximations of
these constants $D_k$ with five significant digits.  Writing
these in terms of the Hardy-Littlewood constants
$$D_k = c_{k-1} \prod_{p < k} \frac{1}{p} \left(\frac{p}{p-1}\right)
^{k-1}$$ we have calculated these with 13 significant digits in
Table \ref{table:D}.

\begin{table}[htbp]
\caption{Adjustment factors $D_{k}$ for arithmetic sequences}
\begin{tabular}{rl|rl} \hline
$k$ & $D_k$ & $k$ & $D_k$ \\
\hline
3 & 1.320323631694 & 12 & 1312.319711299 \\
4 & 2.858248595719 & 13 & 2364.598963306 \\
5 & 4.151180863237 & 14 & 7820.600030245 \\
6 & 10.13179495000 & 15 & 22938.90863233 \\
7 & 17.29861231159 & 16 & 55651.46255350 \\
8 & 53.97194830013 & 17 & 91555.11122614 \\
9 & 148.5516286638 & 18 & 256474.8598544 \\
10 & 336.0343267492 & 19 & 510992.0103092 \\
11 & 511.4222820590 & 20 & 1900972.584874 \\
\hline
\end{tabular}
\label{table:D}
\end{table}

The longest known sequence of arithmetic primes (at the time this
was written) was found in 1993 \cite{PMT95}: it begins with the
prime $11410337850553$ and continues with common difference
$4609098694200$.  Ribenboim \cite[p. 287]{Ribenboim95} has a
table of the first known occurrence of arithmetic sequences of
primes of length $k$ for $12 \le k \le 22$.

\subsection{Evaluating the adjustment factors}

In 1961 Wrench \cite{Wrench61} evaluated the the twin prime
constant with just forty-two decimal place accuracy.  He clearly
did not do this with the product from equation (\ref{eq:C_2})!
Just how do we calculate these adjustment factors (also called
Hardy Littlewood constants and Artin type constants) with any
desired accuracy?  The key is to rewrite them in terms of the
zeta-functions which are easy to evaluate \cite{Bach97,BBC1999}.

Let $P(s)$ be the prime zeta-function:
$$P(s) = \sum_p \frac{1}{p^s}.$$
We can rewrite this using the usual zeta-function
$\zeta(s)=\sum_{n=1}^\infty \frac{1}{n^s}$ and the M\"obius
function $\mu(k)$ as follows (see \cite[pg. 65]{Riesel94}):
$$P(s) = \sum_{k=1}^\infty \frac{\mu(k)}{k} \log \zeta(ks).$$

To evaluate $c_{k}$ we take the logarithm of
equation \ref{eq:c_k} and find
$$\log \left( \prod_{p > k} \frac{1-k/p}{(1-1/p)^k} \right)
  = \sum_{p>k} \left( \log(1-k/p) - k \log(1-1/p) \right).$$
Using the McClaurin expansion for the $\log$ this is
\begin{equation*}
- \sum_{p>k} \sum_{j=1}^\infty \frac{k^j-k}{j p^j} = -
\sum_{j=2}^\infty \frac{k^j-k}{j} \sum_{p>k} p^j = -
\sum_{j=2}^\infty \frac{k^j-k}{j} \left( P(j) - \sum_{p \le k}
p^{-j} \right).
\end{equation*}

It is relatively easy to calculate the zeta-function (see
\cite{BBC1999,Riesel94}), so we now have a relatively easy way to
calculate $c_k$ and $A_{k,d}$.  This approach will easily get
us the fifteen significant decimal place accuracy shown
in table \ref{table:A}.

If we need more accuracy, then we could use the techniques found
in Moree \cite{Moree2000} which
Niklasch\footnote{http://www.gn-50uma.de/alula/essays/Moree/Moree.en.shtml}
used to calculate many such constants with $1000$ decimal places
of accuracy. Moree's key result is that the product
$$C_{f,g}(n) = \prod_{p>p_n}\left( 1-\frac{f(p)}{g(p)} \right)$$
(where $f$ and $g$ are monic polynomials with integer
coefficients satisfying $\deg(f) + 2 \le \deg(g)$ and
$p_n$ is the $n$th prime) can be written as
$$C_{f,g}(n) = \sum_{k=2}^\infty \zeta_n(k)^{-e_k}$$
where the exponents $-e_k$ are integers and $\zeta_n(s) =
\zeta(s) \prod_{p \le p_n}(1-p_n^{-s})$ is the partial zeta
function.

\subsection{Sophie Germain primes}

Recall that $p$ is a Sophie Germain prime if $2p+1$ is also prime
\cite {Yates91}. Therefore, we will use the polynomials $n$ and
$2n+1$. Again, $w(2)=1$ and $w(p)=2$ for all odd primes $p$; so
again our adjustment factor is the twin primes constant $C_{2}$.
This gives us exactly the same estimated number of primes as in
(\ref{pi(twin)}). We can improve this estimate by not replacing
$\log (2n+1)$ with $\log n$ (as we did in (\ref
{p(independent)})). This gives us the following,

\begin{conjecture}
The number of Sophie Germain primes $p$ with $p\leq N$ is
approximately
\begin{equation*}
2C_{2}\int_{2}^{N}\frac{dx}{\log x\log 2x}\sim
\frac{2C_{2}N}{(\log N)^{2}}
\end{equation*}
\end{conjecture}

Again, this estimate (at least the integral) is surprisingly
accurate for small values of $N$, see Table\footnote{Chip
Kerchner provided the last two entries in table
\ref{table:SophieGermain}. (Personal e-mail 25 May 1999.)}
\ref{table:SophieGermain}.

\begin{table}[htbp]
\caption{Sophie Germain primes less than $N$}
\begin{tabular}{rrrr}
\hline & \textbf{actual}
 & \multicolumn{2}{c}{\textbf{predicted}} \\
$N$ & \textbf{number} & \textbf{integral} & \textbf{ratio} \\
\hline
1,000 & \textbf{37} & \textbf{39} & 28 \\
10,000 & \textbf{190} & \textbf{195} & 156 \\
100,000 & \textbf{1171} & \textbf{1166} & 996 \\
1,000,000 & \textbf{7746} & \textbf{7811} & 6917 \\
10,000,000 & \textbf{56032} & \textbf{56128} & 50822 \\
100,000,000 & \textbf{423140} & \textbf{423295} & 389107 \\
1,000,000,000 & \textbf{3308859} & \textbf{3307888} & 3074425 \\
10,000,000,000 & \textbf{26569515} & \textbf{26568824} & 24902848 \\
100,000,000,000 & \textbf{218116524} & \textbf{218116102} & 205808662 \\
\hline
\end{tabular}
\label{table:SophieGermain}
\end{table}

\subsection{Cunningham chains}

Cunningham chains can be thought of as a generalization of Sophie
Germain primes. If the terms of the sequence $$\{ p, 2p+1, 4p+3,
\ldots, 2^{k-1}p+2^{k-1}-1 \}$$ are all prime, then this sequence
is called a Cunningham chain of length $k$. (Sophie Germain primes
yield Cunningham chains of length two.) There is a second type of
these chains, called Cunningham chains of the second kind, which
are prime sequences of the form $$\{p, 2p-1, 4p-3, \ldots,
2^{k-1}p-2^{k-1}+1\}.$$ For either of these forms it is easy to
show that $w(2)=1$, and that for odd primes $p$, $w(p)=\min
(k,ord_{p}(2))$. So the resulting estimate is as follows.

\begin{conjecture}
The number of Cunningham chains of length $k$ beginning with
primes $p$ with $p\leq N$ is approximately
\begin{equation*}
B_{k}\int_{2}^{N}\frac{dx}{\log x\log (2x) \ldots \log (2^{k-1}x)}%
\sim \frac{B_{k}N}{(\log N)^{k}}
\end{equation*}
where $B_{k}$ is the product\
\begin{equation*}
B_{k}=2^{k-1}\prod\limits_{p>2}\frac{p^{k}-p^{k-1}min(k,ord_{p}(2))}{%
(p-1)^{k}}\text{.}
\end{equation*}
\end{conjecture}

\noindent (This conjecture for $k=2, 3$ and $4$, can be found in
\cite{Loh89}.)

 Note that $min(k,ord_{p}(2))$ is just $k$ when
$p>2^{k}$, so we can again write these adjustment factors in
terms of the Hardy-Littlewood constants:
$$B_{k}=2^{k-1} c_k
\prod\limits_{k < p < 2^k}\frac{p-min(k,ord_{p}(2))}{p-k}
\prod\limits_{2 < p \leq
k}\frac{1-min(k,ord_{p}(2))/p}{(1-1/p)^k}.$$ We then count the
Cunningham Chains less than $10^{9}$ as an example to test our
conjecture. As one would expect the agreement is better for the
lower $k$ because these forms yield many more primes for this
small choice of $N$.

%[[[Chris:  Something wrong with this table!  Constants are off]]]

\begin{table}[htbp]
\caption{Cunningham chains of length $k$ starting before $10^9$}
\begin{tabular}{cccccc} \hline \textbf{length} &
adjustment & \multicolumn{2}{c}{\textbf{actual number}}
& \multicolumn{2}{c}{\textbf{predicted}} \\
$k$ & factor $B_{k}$ & \textbf{first kind} & \textbf{second kind}
& \textbf{integral} & \textbf{ratio} \\ \hline 2 & 1.320323631694
& \multicolumn{1}{r}{\textbf{3308859}} &
\multicolumn{1}{r}{\textbf{3306171}} &
\multicolumn{1}{r}{\textbf{3307888}} & \multicolumn{1}{r}{3074426} \\
3 & 2.858248595719 & \multicolumn{1}{r}{\textbf{342414}} &
\multicolumn{1}{r}{\textbf{341551}} &
\multicolumn{1}{r}{\textbf{342313}} & \multicolumn{1}{r}{321163} \\
4 & 5.534907817650 & \multicolumn{1}{r}{\textbf{30735}} &
\multicolumn{1}{r}{\textbf{30962}} &
\multicolumn{1}{r}{\textbf{30784}} & \multicolumn{1}{r}{30011} \\
5 & 20.26358989999 & \multicolumn{1}{r}{\textbf{5072}} &
\multicolumn{1}{r}{\textbf{5105}} &
\multicolumn{1}{r}{\textbf{5092}} & \multicolumn{1}{r}{5302} \\
6 & 71.96222721619 & \multicolumn{1}{r}{\textbf{531}} &
\multicolumn{1}{r}{\textbf{494}} &
\multicolumn{1}{r}{\textbf{797}} & \multicolumn{1}{r}{909} \\
7 & 233.8784426339 & \multicolumn{1}{r}{\textbf{47}} &
\multicolumn{1}{r}{\textbf{46}} &
\multicolumn{1}{r}{\textbf{112}} & \multicolumn{1}{r}{142} \\
\hline
\end{tabular}
\label{table:CuninghamChains}
\end{table}

\subsection{Primes of the form $n^{2}+1$}

 If we use the single polynomial $n^{2}+1$, then
$w(2)=1$, and $w(p)=1+(-1|p)$ for odd primes $p$. Here $(-1|p)$
is the Legendre symbol, so it is $1$ if there is a solution to
$n^{2}\equiv -1 \pmod{p}$, and $-1$ otherwise. Now the adjustment
factor (after a little algebra) is
\begin{equation*}
2\prod\limits_{p>2}1-\frac{(-1|p)}{(p-1)^{2}}=1.3728134628...
\end{equation*}
Calling this constant $C_{+}$ we conjecture that the expected
number of values of $n\leq N$ yielding primes $n^{2}+1$ is
\begin{equation*}
\frac{C_{+}}{2}\int_{2}^{N}\frac{dx}{\log x}\sim \frac{C_{+}}{2}\frac{N}{%
\log N}.
\end{equation*}
But this is not how we usually word our estimates. Often, we
would desire instead the number of primes $n^{2}+1$ that are at
most $N$ (the resulting prime is at most $N$, rather than the
variable $n$ is at most $N$). So we need to replace $N$ by
$\sqrt{N}$ in the integrals' limit, to get:

\begin{conjecture}
The expected number of primes of the form $n^{2}+1$ less than or
equal to $N$ is
\begin{equation}
\frac{C_{+}}{2}\int_{2}^{\sqrt{N}}\frac{dx}{\log x}\sim C_{+}\frac{\sqrt{N}}{%
\log N}.
\end{equation}\label{eq:Cplus}
\end{conjecture}
\noindent (This is \cite[Conjecture E]{HL23}.)

Again these estimates are quite close, see Table
\ref{table:Primesn^2+1}.

\begin{table}[htbp]
\caption{Primes $n^{2}+1$ less than $N$}
\begin{tabular}{rrrr}
\hline
& \textbf{actual} & \multicolumn{2}{c}{ \textbf{predicted}} \\
$N$ & \textbf{number} & \textbf{integral} & \textbf{ratio}
\\ \hline
1,000,000 & \textbf{112} & \textbf{121} & 99 \\
100,000,000 & \textbf{841} & \textbf{855} & 745 \\
10,000,000,000 & \textbf{6656} & \textbf{6609} & 5962 \\
1,000,000,000,000 & \textbf{54110} & \textbf{53970} & 49684 \\
100,000,000,000,000 & \textbf{456362} & \textbf{456404} & 425861 \\
\hline
\end{tabular}
\label{table:Primesn^2+1}
\end{table}

In 1978 Iwaniec showed \cite{Iwaniec78} that there are infinitely
many $P_{2}$'s (products of two primes) among the numbers of the
form $n^{2}+1$.\footnote{He proved that if we divide $C_{+}$ by 77
in equation \ref{eq:Cplus}, then we get a lower bound for the
number of $P_2$'s represented.}  It has also be shown that there
are infinitely many of the form $n^{2}+m^{4}$, but both of these
results are far from proving there are infinitely many primes of
the form $n^{2}+1$.

\addtocontents{toc}{\vspace{5pt}}
\section{Non-polynomial forms}
\addtocontents{toc}{\vspace{3pt}}

In this section we attempt to
apply similar reasoning to non-polynomial forms. There are quite
a few examples of this in the literature:  Mersenne
\cite{Schroeder83, Wagstaff83}, Wieferich \cite{CDP97},
generalized Fermat \cite{DG2000}\footnote{The authors treated
these as polynomials by fixing the exponent and varying the
base.}, primorial and factorial \cite{CG2000}, and primes of the
form $k \cdot 2^{n}+1$ \cite{BR98}.  We will look at several of
these cases below including the Cullen and Woodall primes
(perhaps for the first time).

In the previous sections we took advantage of the fact that for a
polynomial $f(x)$ with integer coefficients, $f(x+p)\equiv
f(x)\pmod{p}$. This is rarely the case when $f(x)$ has a more
general form, and is definitely not true for the form $2^n-1$.  So
rather than use Dickson's Conjecture \ref{KeyConjecture} as we did
in all of the previous sections, we will proceed directly from our
key heuristic: associating with the random number $n$ the
probability $1/\log n$ of being prime--then trying to adjust for
`non-randomness' in each case.

A second common problem we will have with these examples is that
very few primes of each form are known, usually only a couple
dozen at best. When we look back at the numerical evidence for
the polynomial examples, we can not help but notice the
spectacular agreement between the heuristic estimate and the
actual count just begins to show itself after we have many
hundreds, or thousands, of examples.  For that reason it will be
difficult to draw conclusion below from simply counting. We
will also look at the distribution of the know examples and in
some cases the gaps between these examples.

%Why the gaps?  [Chris: do we want to answer this here?]

\subsection{Mersenne primes and the Generalized Repunits}

A repunit is an integer all of whose digits are all one such as
the primes $11$ and $1111111111111111111$. The generalized
repunits (repunits in radix $a$) are the numbers
$R_k(a)=(a^k-1)/(a-1)$.  When $a$ is 2 these are the Mersenne
numbers.  When $a$ is 10, they are the usual repunits.

Before we estimate the number of generalized repunit primes, we
must first consider their divisibility properties.  For example,
if $k$ is composite, then the polynomial $x^k-1$ factors, so for
$R_k(a)$ to be prime, $k$ must be a prime $p$.  As a first
estimate we might guess the probability that $R_k(a)$ is prime is
roughly $(1/\log R_k(a)) (1/\log k) \sim 1/((k-1) \log k \log a)$.

Next suppose that the prime $q$ divides $R_p(a)$ with $p$ prime.
Then the order of $a \pmod{q}$ divides $p$, so is $1$ or $p$.  If
the order is $1$, then $a \equiv 1 \pmod{q}$, $R_p(a) \equiv p$
and therefore $p=q$.  If the order is $p$, then since the order
divides $q-1$, we know $p$ divides $q-1$.  We have shown that
every prime divisor $q$ of $R_p(a)$ is either $p$ (and divides
$a-1$) or has the form $k p+1$ for some integer $k$.

Among other things, this means that for most primes $p$, $R_p(a)$
in not divisible by any prime $q < p$, so we can adjust our
estimate that $R_k(a)$ is prime by multiplying by $1/(1-1/q)$ for
each of these primes.  Here we need to recall an important tool:
\begin{theorem}[Merten's Theorem]
\begin{equation*}
\prod_{\substack{q \leq x \\ q\text{ prime}}}
 \left(1-\frac{1}{q}\right) = \frac{e^{-\gamma}}{\log x} + O(1),
\end{equation*}
\label{thm:Merten}
\end{theorem}
\noindent (For a proof see \cite[p. 351]{HW79}.)  So our second
estimate of the probability that $R_k(a)$ is prime is $e^\gamma
/k \log a$.

\begin{figure}[htbp]
\caption{$\log_2 \log_2 n$th Mersenne prime verses $n$}
\begin{center}
\includegraphics[width=0.55\textwidth]{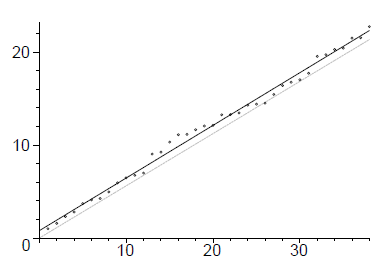}
\end{center}
\label{figure:Mersenne}
\end{figure}

http://www.utm.edu/research/primes/mersenne/heuristic.html

\begin{figure}[htbp]
\caption{$\log_{10} \log_{10} n$th repunit prime verses $n$}
\begin{center}
\includegraphics[width=0.5\textwidth]{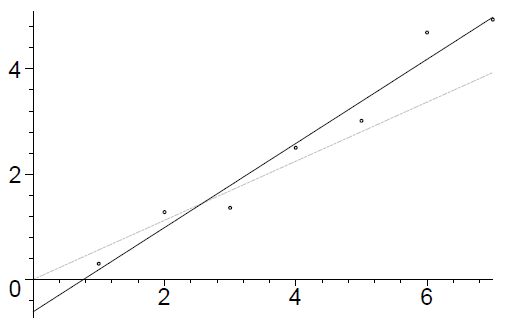}
\end{center}
\label{figure:Repunits}
\end{figure}

\subsection{Cullen and Woodall primes}

The Cullen and Woodall primes are $C(n)=n2^{n}+1$, and
$W(n)=n 2^{n}-1$. In this case we have
\begin{equation*}
C(n+p(p-1))\equiv C(n) \pmod{p(p-1)}
\end{equation*}
and
\begin{equation*}
W(n+p(p-1))\equiv W(n)\pmod{p(p-1)}\text{.}
\end{equation*}
By the Chinese remainder theorem, both of these have $ord_{p}(2)$
solutions in
\begin{equation*}
\{0,1,2,...,p\cdot ord_{p}(2)\},
\end{equation*}
so we might again assume that the probabilities that $p$ divides
$C(n)$ and $W(n)$ are both $1/p$ for odd primes $p$--the same as
for an arbitrary random integers. But are
these probabilities independent for different primes $p$ and $q$?
We must ask this because $p(p-1)$ and $q(q-1)$ are not relatively prime.
We verify this independence as follows:

\begin{theorem}
Let $p$ and $q$ be distinct odd primes and let $a$ and $b$ be any
integers. The the system of congruences
\begin{equation*}
\left\{
\begin{array}{c}
n2^{n}\equiv a \pmod{p} \\
n2^{n}\equiv b \pmod{q}
\end{array}
\right.
\end{equation*}
has $lcm(pq,ord_{p}(2),ord_{q}(2))/pq$ solutions in $\{0,1,2,...,%
lcm(pq,ord_{p}(2),ord_{q}(2))\}$.

\begin{proof}
For each $r$ modulo $d=lcm(ord_{p}(2),ord_{q}(2))$ write
$n=r+sd$. Then the system above is
\begin{equation*}
\left\{
\begin{array}{c}
sd\equiv a/2^{r}-r \pmod{p} \\
sd\equiv b/2^{r}-r \pmod{q}
\end{array}
\right.
\end{equation*}
Assume that $q$ is the larger of the two primes, then we know
$q\nmid ord_{p}(2)$, so the second of these congruences has a
unique solution (modulo $q$).  If $p\nmid d$, then the first
congruence also has a unique solution, giving a total of $d$
solution to the original system (one for each $r$).  In this case
$d$ is $lcm(pq,ord_{p}(2),ord_{q}(2))/pq$. On the other hand, if
$p\mid d$, then the only acceptable choices of $r$ are those for
which $r2^{r}\equiv a \pmod{p}$.  There are $d/p$ of these--which
again is $lcm(pq,ord_{p}(2),ord_{q}(2))/pq$.
\end{proof}
\end{theorem}

For each odd prime the analog of the adjustment factor
(\ref{factor for prime}) is therefore one, and the complete
adjustment factor (\ref {adjustment factor}) is $2$ in both cases
(Cullen and Woodall).  This gives us the following.

\begin{conjecture}
The expected number of Cullen and Woodall primes with $n\leq N$
are each \
\begin{equation}
2\int_{2}^{N}\frac{dx}{\log x2^{x}}\backsim 2\frac{\log N-\log
2}{\log 2} \label{bad_est}
\end{equation}
\end{conjecture}

Table \ref{table:CullenWoodall} shows us that what little
evidence we have does not support (\ref{bad_est}) well for the
Cullen numbers, though it does appear reasonable for Woodall
numbers.

\begin{table}[htbp]
\caption{Cullen $C(n)$ and Woodall $W(n)$ primes}
\begin{tabular}{rrrcc}
\hline
& \multicolumn{2}{r}{\textbf{actual }(with $n<N)$} & \multicolumn{2}{c}{%
\textbf{predicted}} \\
$N$ & \textbf{Woodall} & \textbf{Cullen} & \textbf{integral} &
\textbf{ratio} \\ \hline
1000 & \textbf{15} & \textbf{2} & \textbf{15} & 18 \\
10,000 & \textbf{18} & \textbf{5} & \textbf{22} & 25 \\
100,000 & \textbf{24} & \textbf{10} & \textbf{29} & 31 \\
500,000 & $\geq $\textbf{26} & $\geq $\textbf{13} & \textbf{33} & 36 \\
1,000,000 &  &  & \textbf{35} & 38 \\ \hline
\end{tabular}
\label{table:CullenWoodall}
\end{table}

%\begin{quotation}
%[[ Chris: Why is this so bad?  Keller's article lists the main
%divisibility properties of these numbers, perhaps we should work
%it in! ]]
%\end{quotation}

Since we have so few data points it might offer some insight to
graph the expected number of Cullen and Woodall primes below each
of the known primes of these forms (see graph ?removed?).  If our
estimate holds, then this graph would remain ``near'' the
diagonal.

\subsection{Primorial primes}

Primes of the form $p\# \pm 1$ are sometimes called the primorial
primes (a term introduced by H. Dubner as a play on the words
prime and factorial). Since $\log p\#$ is the Chebyshev theta
function, it is well known that asymptotically $\theta (p) = \log
p\#$ is approximately $p$. In fact Dusart \cite{Dusart99} has
shown

\begin{equation*}
\abs{\theta (x) - x} \leq 0.006788 \frac{x}{\log x} \text{ \ \ \
\ \ \ for }x \geq 2.89 \times 10^{7}.
\end{equation*}

\noindent We begin (as usual) noting that by the prime number
theorem the probability of a ``random'' number the size of $p\#
\pm 1$ being prime is asymptotically $\frac{1}{p}$. However, $p\#
\pm 1$ does not behave like a random variable because primes $q$
less than $p$ divide $1/q^{\text{th}}$ of a random set of
integers, but can not divide $p\# \pm 1$. So we adjust our
estimate by dividing by $1 - \frac{1}{q}$ for each of these
primes $q$. By Mertens' theorem \ref{thm:Merten} our final
estimate of the probability that $p\# \pm 1$ is prime is
$\frac{e^{\gamma} \log p}{p}$.

By this simple model, the expected number of primes $p\# \pm 1$
with $p \leq N$ would then be

\begin{equation*}
\sum_{p \leq N} \frac{e^{\gamma} \log p}{p} \sim e^{\gamma} \log N
\end{equation*}

\begin{conjecture}
The expected number of primorial primes of each of the forms $p\#
\pm 1$ with $p \leq N$ are both approximately $e^{\gamma} \log N$.
\end{conjecture}

The known, albeit limited, data supports this conjecture. What is
known is summarized in Table~\ref{tab:3}.

\begin{table}[htbp]
\caption{The number of primorial primes $p\#\ \pm 1$ with $p \le
N$} \label{tab:3}
\begin{tabular}{rccc}
\hline
 & \multicolumn{2}{c}{\textbf{actual}} & \textbf{predicted} \\
$N$ & $p\# + 1$ & $p\# - 1$ & (of each form) \\
\hline
10 & \textbf{4} & \textbf{2} & \textbf{4} \\
100 & \textbf{6} & \textbf{6} & \textbf{8} \\
1000 & \textbf{7} & \textbf{9} & \textbf{12} \\
10000 & \textbf{13} & \textbf{16} & \textbf{16} \\
100000 & \textbf{19} & \textbf{18} & \textbf{20} \\
\hline
\end{tabular}
\end{table}

\begin{remark}
By the above estimate, the $n^{\text{th}}$ primorial prime should
be about $e^{n/e^{\gamma}}$.
\end{remark}

\subsection{Factorial primes}

The primes of the forms $n! \pm 1$ are regularly called the
factorial primes, and like the ``primorial primes'' $p\# \pm 1$,
they may owe their appeal to Euclid's proof and their simple form.
Even though they have now been tested up to $n=10000$
(approximately 36000 digits), there are only $39$ such primes
known. To develop a heuristical estimate we begin with Stirling's
formula:

\begin{equation*}
\log n! = (n+\frac{1}{2}) \log n - n + \frac{1}{2} \log 2 \pi +
O\left(\frac{1}{n}\right)
\end{equation*}

\noindent or more simply: $\log n! \sim n(\log n - 1)$. So by the
prime number theorem the probability a random number the size of
$n! \pm 1$ is prime is asymptotically $\frac{1}{n(\log n - 1)}$.

Once again our form, $n! \pm 1$ does not behave like a random
variable--this time for several reasons.
First, primes $q$ less than $n$ divide $1/q$--th
of a set of random integers, but can not divide $n! \pm 1$. So we
again divide our estimate by $1 - \frac{1}{q}$ for each of these
primes $q$ and by Mertens' theorem we estimate the
probability that $n! \pm 1$ is prime to be

\begin{equation}
\label{eq:FactEst} \frac{e^{\gamma} \log n}{n(\log n - 1)}.
\end{equation}

To estimate the number of such primes with $n$ less than $N$, we
may integrate this last estimate to get:

\begin{conjecture}
The expected number of factorial primes of each of the forms $n!
\pm 1$ with $n \le N$ are both asymptotic to $e^{\gamma} \log N$
\end{conjecture}

Table~\ref{tab:2} shows a comparison of this estimate to the
known results.

\begin{table}[htbp]
\caption{The number of factorial primes $n! \pm 1$ with $n \le N$}
\label{tab:2}
\begin{tabular}{rccc}
\hline
 & \multicolumn{2}{c}{\textbf{actual}} & \textbf{predicted} \\
$N$ & $n! + 1$ & $n! - 1$ & (of each form) \\
\hline
10 & \textbf{3} & \textbf{4} & \textbf{4} \\
100 & \textbf{9} & \textbf{11} & \textbf{8} \\
1000 & \textbf{16} & \textbf{17} & \textbf{12} \\
10000 & \textbf{18} & \textbf{21} & \textbf{16} \\
\hline
\end{tabular}
\end{table}

As an alternate check on this heuristic model, notice that it also
applies to the forms $k\cdot n!\pm 1$ ($k$ small). For $1\le k\le
500$ and $1\le n\le 100$ the form $k\cdot n!+1$ is a prime 4275
times, and the form $k\cdot n!-1$, 4122 times.  This yields an
average of 8.55 and 8.24 primes for each $k$, relatively close to
the predicted $8.20$.

But what of the other obstacles to $n! \pm 1$ behaving randomly?
Most importantly, what effect does accounting for Wilson's theorem
have? These turn out not to significantly alter our estimate
above.  To see this we first we summarize these divisibility
properties as follows.

\begin{theorem}
\label{thm:FactDiv} Let $n$ be a positive integer.
\begin{enumerate}
  \item[i)] $n$ divides $1!-1$ and $0!-1$.
  \item[ii)] If $n$ is prime, then $n$ divides both $(n-1)!+1$ and $(n-2)!-1$.
  \item[iii)] If $n$ is odd and $2n+1$ is prime, then $2n+1$ divides
  exactly one of $n!\pm 1$.
  \item[iv)] If the prime $p$ divides $n!\pm 1$, then $p-n-1$ divides one of $n!\pm
  1$.
\end{enumerate}
\end{theorem}

\begin{proof}
(ii) is Wilson's theorem.  For (iii), note that if $2n+1$ is
prime, then Wilson's theorem implies $$-1 \equiv
1\cdot2\cdot\ldots\cdot n\cdot (-n)\cdot\ldots\cdot (-1) \equiv
(-1)^n (n!)^2 \pmod{2n+1}.$$  When $n$ is odd this is $(n!)^2
\equiv 1$, so $n! \equiv \pm 1 \pmod{2n+1}$. Finally, to see (iv),
suppose $n!\equiv \pm 1 \pmod{p}$. Since $(p-1)!\equiv -1$, this
is $$(p-1)(p-2)\cdot \ldots \cdot (n+1) \equiv
(-1)^{p-n-1}(p-n-1)! \equiv \mp 1 \pmod{p}.$$  This shows $p$
divides exactly one of $(p-n-1)!\pm 1$.
\end{proof}

To adjust for the divisibility properties (ii) and (iii), we
should multiply our estimate \ref{eq:FactEst} by $1 -
\frac{1}{\log n}$, which is roughly the probability that $n+1$ or
$n+2$ is composite; and then by $1 - \frac{1}{4\log 2n}$, which
is the probability $n$ is odd and $2n+1$ is prime.  The other two
cases of Theorem~\ref{thm:FactDiv} require no adjustment.  This
gives us the following estimate of the primality of $n!\pm 1$.

\begin{equation}
\left(1 - \frac{1}{4 \log 2n}\right)\frac{e^\gamma}{n}
\end{equation}

\noindent Integrating as above suggests there should be
$e^{\gamma} (\log N - \frac{1}{4} \log \log 2N )$ primes of the
forms $n! \pm 1$ with $n \le N$.  Since we are using an integral
of probabilities in our argument, we can not hope to do much
better than an error of $o(\log N)$, so this new estimate is
essentially the same as our conjecture above.

\addtocontents{toc}{\vspace{5pt}}

%\bibliographystyle{amsplain}
%\bibliography{primes.bib}
\providecommand{\href}[2]{#2 ({\it #1})} \providecommand{\url}[1]{{\it #1}}
\providecommand{\bysame}{\leavevmode\hbox to3em{\hrulefill}\thinspace}
\providecommand{\MR}{\relax\ifhmode\unskip\space\fi MR }
\providecommand{\MRhref}[2]{%
  \href{http://www.ams.org/mathscinet-getitem?mr=#1}{#2}
}
\providecommand{\href}[2]{#2}

\end{document}